\documentclass[manyauthors]{fundam}

\usepackage{amsfonts,amssymb}

\newtheorem{thm}{Theorem}[section]
\newtheorem{lem}[thm]{Lemma}
\newtheorem{prop}[thm]{Proposition}
\newtheorem{defin}[thm]{Definition}
\newtheorem{rem}[thm]{Remark}

\newcommand{\GN}{\mathbb{N}}

\DeclareMathOperator{\alf}{Alph}
\DeclareMathOperator{\Ff}{Fact}
\DeclareMathOperator{\suc}{succ}
\newcommand{\llex}{<_{\mathrm{lex}}}

\title{Some characterizations of Sturmian words in terms of the lexicographic order}

\author{Michelangelo Bucci\thanks{Partially supported by the FiDiPro grant ``Words, Numbers, Tilings and Applications'' from the Academy of Finland}\\
Department of Mathematics,
University of Turku\\
FI-20014 Turku, Finland\\
michelangelo.bucci{@}utu.fi
\and
Alessandro De Luca\thanksas{1}\\
Dipartimento di Scienze Fisiche\\
Universit\`a degli Studi di Napoli Federico II\\
via Cintia, Monte S.~Angelo\\
I-80126 Napoli, Italy\\[1ex]
Department of Mathematics, University of Turku\\
FI-20014 Turku, Finland\\
alessandro.deluca{@}utu.fi
\and
Luca Q.~Zamboni\thanksas{1}\thanks{Partially supported by ANR grant \textsl{SUBTILE}}\\
Universit\'e de Lyon, 
Universit\'e Lyon 1\\ 
CNRS UMR 5208 Institut Camille Jordan\\ 
B\^atiment du Doyen Jean Braconnier\\ 
43, blvd du 11 novembre 1918\\ 
F-69622 Villeurbanne Cedex, France\\[1ex]
Department of Mathematics, University of Turku\\
FI-20014 Turku, Finland\\
luca.zamboni{@}utu.fi}

\begin{document}

\maketitle

\runninghead{M.~Bucci, A.~De Luca, L.~Q.~Zamboni}{Some characterizations of Sturmian words in terms of the lex order}
\begin{abstract}
In this paper we present three new characterizations of Sturmian words based on the lexicographic ordering of their factors. 
\end{abstract}

\section{Introduction}
Let $w \in A^{\omega} $ be an infinite word with values in a finite alphabet
$A$.  The {\it (block) complexity function} $p_{w }:\GN \rightarrow \GN$ assigns
to each $n$ the number of distinct factors of $w $ of length $n$. A fundamental result due to Hedlund
and Morse \cite{MoHe1} states that a word $w $ is ultimately periodic if and only if for some $n$ the complexity $p_{w }(n)\leq n$.  Sequences of complexity $p(n)=n+1$ are called 
{\it Sturmian words.} The most studied Sturmian word is the so-called Fibonacci word 
\[01001010010010100101001001010010\ldots\]
fixed by the morphism $0\mapsto 01$ and $1\mapsto 0$. In \cite {MoHe2} Hedlund and Morse showed that each Sturmian word may be realized geometrically by an irrational rotation on the circle. More precisely, every Sturmian word is obtained by coding the symbolic orbit of a point $x$ on the circle (of circumference one) under a rotation by an irrational angle $\alpha $ where the circle is partitioned
into two complementary intervals, one of length $\alpha $ and the other of length $1-\alpha $.
And conversely each such coding gives rise to a Sturmian word. The irrational $\alpha $ is called the {\it slope} of the Sturmian word. An alternative characterization using continued fractions was given by Rauzy in \cite{Ra1} and \cite{Ra2}, and later by Arnoux and Rauzy
in \cite {ArRa}. Sturmian words admit various other types of characterizations of geometric and combinatorial nature (see for instance \cite{BeSe}).
For example they are characterized by the following balance property: A word $w $
is Sturmian if and only if $w $ is a binary aperiodic (non-ultimately periodic) word with the property that
for any two factors $u$ and $v$ of $w $ of equal length, we have $-1\leq |u|_i-|v|_i\leq 1$ for each letter $i$. Here $|u|_i$ denotes the number of occurrences of $i$ in $u$.
 In this paper, we establish some new characterizations of Sturmian words in terms of  the lexicographic order behavior of its factors. We prove:
 
\begin{thm}
\label{theo1}
An infinite word $w$ containing the letters $0$ and $1$ is Sturmian if and only if for every pair of lexicographically consecutive factors $v,v'$ of the same length,   there exist $\lambda, \mu$ such that $v,v'$ either both belong to $\{\lambda 01 \mu, \lambda 10 \mu\}$ or both belong to  $\{ \lambda 0, \lambda 1\}$.
\end{thm}
Actually our first main result is later formulated in more general terms.
The fact that this property holds for Sturmian words has recently been shown in~\cite{PeRe}, and is a direct consequence of a result proved in \cite{bib:ZamJen}. 

Our second characterization requires the additional hypothesis of recurrence:

\begin{thm}\label{theo2}
Let $w$ be a recurrent aperiodic binary word over the alphabet $\{ 0, 1 \}$ and $v, v' \in \Ff(w)$. Then the following are equivalent:
\begin{enumerate}
\item $w$ is Sturmian.
\item For all factors $v,v'$ of $w$ of equal length, if $v \llex v'$ then $|v|_{1} \leq |v'|_{1}$.
\item For any pair of lexicographically consecutive factors $v,v'$ of the same length, $v$ and $v'$ differ in at most two positions. 
\end{enumerate}
\end{thm}


\section{Preliminaries}\label{sec:prel}

In this section, we introduce the tools which will be used in the rest of the paper. 

\subsection{Standard notions in combinatorics on words}
We will report here the standard notations and notions in combinatorics on words that will be used in the rest of the paper. For further results on the subject we refer the reader to  \cite{bib:lot}.

By an \emph{alphabet} we mean a finite non empty set $A$. The elements of $A$ are called \emph{letters}. We let $A^{*}$ denote the free monoid over $A$, i.e.~the set of finite sequences of elements of $A$ equipped with the concatenation product. The neutral element of $A^{*}$ will be called the \emph{empty word} and is denoted $\varepsilon$. The set of nonempty words over $A$, i.e.~the free semigroup over $A$, is denoted $A^{+}$. With the multiplicative notation, given a positive integer $n$ and a word $w$, we let $w^{n}$ denote the concatenation of $n$ copies of $w$. For each word $w$, we put $w^{0}= \varepsilon$.

Two words $v, v'$ are said to be \emph{conjugates} one of the other if there exist $\lambda, \mu$ such that $v= \lambda \mu$ and $v'= \mu \lambda$.

If a nonempty word $x$ is such that $x=x_{1}x_{2}\cdots x_{k}$, with $x_{i}\in A$ for $1\leq i\leq k$, then $k$ is called the  \emph{length} of $x$ and is denoted $|x|$. The length of the empty word is taken to be $0$. 

We say that a word $v$ is a \emph{factor} of another word $w$ if there exist two words $\lambda, \mu$ such that $w=\lambda v \mu$. If $\lambda= \varepsilon$ (resp. $\mu=\varepsilon$) we call $v$ a \emph{prefix} (resp.~a \emph{suffix}) of $w$. If $v$ is both a prefix and a suffix of $w$, we say that $v$ is a \emph{border}. A factor $v$ of $w$ is called \emph{proper} if $|v| < |w|$. We denote with $\Ff(w)$ the set of all factors of the word $w$. A word $w$ is said to be \emph{unbordered} if the only borders of $w$ are $w$ and $\varepsilon$.

Most of the above definitions can be extended to the set $A^{\omega}$ of infinite words on the alphabet $A$. For $w, w'\in A^{\omega}$, we say $w'$ is a \emph{tail} of $w$ if  $w=vw'$ for some $v\in A^*$. If $v$ is not empty, we call $w'$ a \emph{proper tail} of $w$.

We call an \emph{occurrence} of $v$ in $w$ a word $\lambda$ such that $\lambda v$ is a prefix of $w$. An infinite word $w$ is said to be \emph{recurrent} if each of its factors (or, equivalently, of its prefixes) has infinitely many occurrences in $w$.
Given $v,w\in A^*$ we let $|w|_{v}$ denote the number of occurrences of $v$ in $w$ and set 
\[\alf(w) = \{ x \in A \, | \, |w|_{x}> 0 \}.\]
A factor $v$ of $w$ is \emph{unioccurrent} if $|w|_{v}=1$, i.e., if $v$ occurs in $w$ exactly once.

We say that an infinite word $w$ is \emph{periodic} if it can be expressed as an infinite concatenation of a finite word $v$, i.e. $w=v^{\omega}$
. We say that an infinite word is \emph{ultimately periodic} 
if it has a periodic 
tail. Otherwise we say $w$ is \emph{aperiodic}.  It is easy to show that any infinite word that contains itself as a proper tail is periodic.

\subsection{Lexicographic order}

Let $A$ be an alphabet equipped with a total order $<. $ Then $<$ extends naturally to a partial order on $A^{*}$, denoted  $\llex$, in the following way: We write $v \llex v'$ (and say  $v$ is \emph{lexicographically} smaller than $v',)$   if $|v|=|v'|$ and there exists a word $\lambda$ and two letters $a < b$ such that $\lambda a$ is a prefix of $v$ and $\lambda b$ is a prefix of $v'$. Two words $v,v'$ are said to be lexicographically \emph{consecutive} or \emph{adjacent} if $v \llex v'$ and there is no word $w$ such that $v \llex w \llex v'$.

We say a factor $v$ of a word $w$ is \emph{maximal} (resp. \emph{minimal}) in $w$ if there exists no factor $v'$ such that $v \llex v'$ (resp. $v' \llex v$), thus omitting the sentence ``with respect to the lexicographic order''. We will say that $v$ is \emph{extremal} in $w$  if it is either minimal or maximal.

Given two factors $v,v'$ of a word $w$ such that $v \llex v'$, we will write \( v' = \suc_{w}(v) \) if there is no $f \in \Ff(w)$ such that $v \llex f \llex v'$.  Notice that if $v \in \Ff(w)$ is non extremal, then there exist $f_{1}, f_{2} \in \Ff(w)$ such that $f_{1} = \suc_{w}(v)$ and $v = \suc_{w}(f_{2})$.

\begin{rem}\label{rem:extremalpref}
It is easy to show that if $v$ is a unioccurrent prefix of an infinite word $w$ and $v$ is extremal in $w$, then every prefix of $w$ longer than $v$ is unioccurrent, and extremal of the same kind.
\end{rem}

We can extend the definition of lexicographic order to infinite words in a natural way, saying that the infinite word $w$ is lexicographically smaller than $w'$ if $w$ has a prefix which is lexicographically smaller than a prefix (of the same length) of $w'$. The notion of extremality extends as well: we say that an infinite word $w$ is minimal (resp.~maximal) if it is lexicographically smaller (resp.~larger) than all its tails. 

\begin{rem}\label{rem:extremaltail}
It is clear that if $aw$ and $w$ are both extremal infinite binary words (and $a$ is a letter), then they are extremal of the same kind (i.e.~they are both minimal or maximal).
\end{rem}

\subsection{Sturmian words}
Let $v$ and $v'$ be factors of $w$ with $|v|=|v'|$. We say the pair $(v,v')$ is \emph{balanced} if $\left | |v|_{x}-|v'|_{x} \right|\leq 1$ for each letter $x\in A$. Otherwise the pair $(v,v')$ is said to be \emph{imbalanced}. A word $w$ is called balanced if all pairs of factors of $w$ of the same length are balanced.

A binary word $w$ is called \emph{Sturmian} if $w$ is aperiodic and balanced. As mentioned earlier, Sturmian words are also defined in terms of the block complexity function $p_{w }:\GN \rightarrow \GN$ which assigns
to each $n$ the number of distinct factors of $w $ of length $n$: $w$ is Sturmian if and only if $p_w(n)=n+1$ for each $n\geq 0$.  

For each Sturmian word $w\in \{0,1\}^\omega$ we set
\[ \Omega_{w} = \{w'\in \{0,1\}^\omega \, | \, \Ff(w') = \Ff(w) \}. \]
Thus  $ \Omega_{w}$ is the \emph{shift orbit closure} or \emph{subshift} generated by $w$.  
The proof of the following proposition is in  \cite{BeSe}.

\begin{prop} \label{thm:sturm}
Let $w$ be a Sturmian word over the alphabet $\{0,1\}$. Then
 there exists a unique word $\gamma $ in  $\Omega_{w}$ such that both $0\gamma$ and $ 1\gamma$ are in  $\Omega_{w}$.
\end{prop}

\begin{rem}\label{rem:sturm}
The word $\gamma$ in Proposition \ref{thm:sturm} is called the \emph{characteristic} word of $w$ and it is known that the prefixes of $0\gamma$ are lexicographically minimal among the factors of $w$, while the prefixes of $1\gamma$ are maximal.
\end{rem}

We say that a factor $v$ of a Sturmian word $w$ is a \emph{Christoffel word} if $v$ is unbordered. We group into the next statement the well-known properties of Christoffel words that we will need in the rest of the paper (see for instance~\cite{MigZam,blrs}, \cite[Prop.~5]{CaoWen}, \cite[Prop.~6]{bib:ZamJen}). 

\begin{prop} \label{thm:christoffel}
Let $w$ be a Sturmian word over the alphabet $\{ 0,1 \}$ and let $v \in \Ff(w)$ be a Christoffel word such that $|v|>1$. Then there exists $u$ such that:
\begin{enumerate}
\item $v$ is either $0u1$ or $1u0$, and they are both Christoffel words in $w$;
\item $0u1$ and $1u0$ are the only Christoffel words of length $|v|$ in $w$ and are conjugates;
\item all conjugates of $v$ are in $\Ff(w)$;
\item exactly one between $0u0$ and $1u1$ is a factor of $w$ and is extremal in $w$;
\item the factors of $w$ of length $|v|$ are either conjugates of $v$ or of type $xux$. 
\end{enumerate}

\end{prop}

Those factors of a Sturmian word having the same length as a Christoffel word, but  not conjugate to a Christoffel word (i.e.~the factors $xux$ in the preceding proposition), are called \emph{singular words} of the Sturmian word.

\section{Main Result}\label{sec:main}
We begin with the following key proposition: 

\begin{prop}\label{thm:nonsturm}
Let $w\in \{0,1\}^\omega$ be an imbalanced word. Then there exists a factor $u \in \Ff(w)$ of minimal length such that $0u0, 1u1$ are in $\Ff(w)$.
Furthermore, either $10u0$ and $01u1$ are both factors of $w$ or there exists a unique letter $x$ such that $xux$ is a prefix of $w$ and occurs in $w$ only finitely many times. In the latter case every prefix of $w$  is extremal in $w$.
\end{prop}

\begin{proof}
Since $w$ is not balanced, there exists an imbalanced pair $(v,v')$ consisting of factors $v$ and $v'$ of $w$. It is well known (see \cite{BeSe}) that the imbalanced pair of minimal length is of the form $(0u0,1u1)$ for some factor $u$ of $w$ and is unique. If both $10u0$, $01u1$ are factors of $w$ we are done.  So let us assume that there exists a letter $x\in \{0,1\}$ such that no occurrence of $xux$ in $w$ is preceded by $1-x$. Then every internal (non-prefix) occurrence of $xux$ in $w$ is preceded by $x$. We begin by showing that  $xux$ is a prefix of $w$ from which it follows that $x$ is unique. Without loss of generality we can assume that $x=0$. Suppose that the first occurrence of $0u0$ in $w$ occurs in position $n\geq 0$. If $n=0$ we are done. So suppose $n>0$.  Then $00u$ is a factor of $w$ occurring in position $n-1$ and the pair $(00u, 1u1)$ is imbalanced. By uniqueness of the shortest imbalanced pair we have that $00u=0u0$ and hence  $0u0$ also occurs in position $n-1$, a contradiction on the minimality of $n$. This also shows that if $0u0$ occurs in position $t$ then it also occurs in each position $r$ for $0\leq r\leq t$. Thus $0u0$ occurs only finitely many times in $w$ (for otherwise $w$ would be $0^{\omega}$ and thus not binary).

We next show that every prefix of w is minimal (if we had taken $x=1$ then each prefix of $w$ would be maximal). 
We proceed by contradiction. Let $n>0$ be the least positive integer for which there exists a factor $v'$ of $w$ in position $n$ which is lexicographically smaller than the corresponding prefix $v$ of $w$ of the same length. Then either there exists a proper prefix $u'$ of $u$ such that $0u'1$ is a prefix of $v$ and $0u'0$ is a prefix of $v'$, or $v'$ begins in $0u0$. In the first case  $0u'0$ and the prefix $1u'1$ of $1u1$ constitute a shorter imbalanced pair contradicting the minimality of $|u|$. In the second case $v'$ is an internal occurrence of $0u0$ and is hence
preceded by $0$. Thus the factor $v''$ in position $n-1$ of length $|v'|$ is lexicographically smaller than $v'$ and also smaller than $v$, contradicting the minimality of $n$. 
\end{proof}

The next proposition introduces the main subject of this paper:

\begin{prop}\label{thm:NFOp}
Fix $k\geq 1$. Let  $A=\{0,1, \ldots , k\}$ be an ordered alphabet such that $0 < 1 < \cdots < k$ and $w$ an infinite word such that $\alf(w) = A$.
The following are equivalent:
\begin{enumerate}
\item For every $v,v' \in \Ff(w)$ with $v'= \suc_{w}(v)$, there exist distinct letters $a < b$ in $A$ and $ \lambda, \mu \in A^{*}$ such that
\[ \left\{  
\begin{array}{lcl} 
v &=& \lambda ab \mu \\  v' &=& \lambda ba \mu \end{array} \right.
\quad \text{OR} \quad
\left\{ \begin{array}{lcl} v &=& \lambda a \\ v' &=& \lambda b\end{array} \right. \]

\item For every $v,v' \in \Ff(w)$ with $v'= \suc_{w}(v)$, there exist $m\in A$ and $ \lambda, \mu \in A^{*}$ such that
\[ \left\{  
\begin{array}{lcl} 
v &=& \lambda m(m+1) \mu \\  v' &=& \lambda (m+1)m \mu \end{array} \right.
\quad \text{OR} \quad
\left\{ \begin{array}{lcl} v &=& \lambda m \\ v' &=& \lambda (m+1)\end{array} \right. \]

\item $A=\{0,1\}$ and  for every $v,v' \in \Ff(w)$ with $v'= \suc_{w}(v)$, there exist $ \lambda, \mu \in A^{*}$ such that
\[ \left\{  
\begin{array}{lcl} 
v &=& \lambda 01 \mu \\  v' &=& \lambda 10 \mu \end{array} \right.
\quad \text{OR} \quad
\left\{ \begin{array}{lcl} v &=& \lambda 0 \\ v' &=& \lambda 1\end{array} \right. \]
\end{enumerate}
\end{prop}

\begin{proof}
Clearly $(3) \Rightarrow (2) \Rightarrow (1)$. To see that $(1) \Rightarrow (3)$ it suffices to show that $(1)$ implies that $k=1$.  
We first note that if  $ab \in \Ff(w)$, then $b\in \{a-1,a,a+1\}$.  In fact suppose that $a\neq b$. Then either $a<b$ or $b<a$. We consider the first case as the latter case is essentially identical. Let $x,y\in A$ such that $ax$ is the greatest factor of length $2$  beginning with $a$ and $(a+1)y$ be the smallest factor of length 2 beginning with $(a+1)$. Clearly $(a+1)y = \suc_{w}(ax)$, which, from the hypothesis implies that $x=a+1$ and $y=a$. Thus $ab$ is lexicographically smaller or equal to $ a(a+1) $ from which it follows that $b=a+1$.

Now suppose to the contrary that $k>1$, and consider the shortest factor $v$ of $w$ containing both $0$ and $2$.
Then, from what we just proved, $v=01^n2$ or $v=21^n0$ for some $n>0$.  We will show that neither occurs in $w$. Suppose to the contrary that the first is a factor of $w$ and consider the least $n> 0$ for which $01^n2$ is a factor of $w$.  Then as $01^n2$ is not maximal,  its successor is either of the form $101^{n-1}2$ or $ 01^{n-1}21$ or $01^nx$ for some $2<x$.  The first two cases contradict the minimality of $n$ while the last case implies that $1x$ is a factor of $w$ for some $2<x$, again a contradiction. Similarly it is verified that  $v=21^n0$ is never a factor of $w$.
Hence $k=1$. \end{proof}

\begin{defin}
We say that an infinite word $w$ has the ``Nice Factors Ordering property'' (NFOp) if for $w$ one of the equivalent conditions of Proposition \ref{thm:NFOp} holds.
\end{defin}

\begin{rem}\label{rem:binary}
It is useful to stress that having the NFOp implies that the word $w$ is actually binary. Also it is easy to see that
NFOp actually characterizes the pairs of adjacent factors with respect to the lexicographic ordering, i.e., 
If $w$ satisfies NFOp and  $v$ and $v'$ are factors of $w$ with $v= \lambda 01 \mu$ and $ v'=\lambda 10 \mu$ or $v= \lambda 0$ and $v' = \lambda 1$, then $v' = \suc_{w}(v)$.
\end{rem}

\begin{lem}\label{thm:nonperiodic}
If an infinite word $w$ has the NFOp, then $w$ is aperiodic.
\end{lem}

\begin{proof}
Let as assume by contradiction that there exist $w', v \in A^{*}$  with $w=w'v^{\omega}$. Then $w$ has finitely many tails and it is readily proved that they must respect the NFOp, i.e.~if $x$ and $y$ are two lexicographically consecutive tails of $w$ then we can write
\[ x=z01z' \quad y=z10z'. \]
In particular, this implies that every tail contains either $01$ or $10$, hence $v$ cannot be a single letter.
As $v^{\omega}$ contains both $01$ and $10$,  $v^{\omega}$ contains  tails of the form  $(01v')^{\omega}$ and $\mu=(10v'')^{\omega}$ for some 
$v', v''$with $|v'|=|v''|=|v|-2$.  Clearly these two tails differ in an infinite number of positions. On the other hand
$w$ has only a finite number of tails and by assumption any two lexicographically consecutive tails differ in exactly two positions. Hence we obtain a contradiction. 
\end{proof}

\begin{lem}\label{thm:contradiction}
Let $w$ be an infinite word with the NFOp. Then there exists no factor $u$ in $\Ff(w)$ such that $10u0$ and $01u1$ are both factors of $w$.
\end{lem}

\begin{proof}
Suppose to the contrary that there exists a shortest factor $u$ such that both $10u0$ and $01u1$ are factors of $w$. Since $01u1 \llex 10u0$, but the two factors cannot be consecutive as they differ in at least three positions, the successor  $v$ of  $01u1$ satisfies $ 01u1 \llex v \llex 10u0$. It follows  that there exists a proper prefix $\lambda$ of $u$ such that $01\lambda 0$ is a prefix of $01u1$ and $01\lambda 1$ is a prefix of $v$ (notice that $v$ cannot begin with $1$ since otherwise it would be $10u1$ and thus would be lexicographically larger than $10u0$). Since $10\lambda 0$ is a prefix of $10u1$ the factors  $01\lambda 1$ and $10\lambda 0$ contradict the minimality of $|u|$. 
\end{proof}

The following result is a direct consequence of a result proved by the third author together with Jenkinson in \cite{bib:ZamJen} and, more recently, has appeared in \cite{PeRe}; we include it here with a different proof, for the sake of completeness. 

\begin{prop}\label{thm:firstDirection}
Let $w$ be a Sturmian word on the alphabet $\{ 0,1\}$. Then $w$ satisfies NFOp.
\end{prop}

\begin{proof} Let  $0u1$ be a Christoffel factor of $w$. As $1u0$ is a conjugate of $0u1$ it follows that each factoring $u=xy$ determines two conjugates of $0u1$, namely $v=y01x$ and $v'=y10x$.
By Proposition~\ref{thm:christoffel}, $v$ and $v'$ are factors of $w$; 
let $z=\suc_{w}(v)$. As $v\llex v'$, the longest common prefix of $v$ and $z$ is at least $y$. In fact it cannot be longer, otherwise we could write $v=y01x'0\lambda$ and $z=y01x'1\mu$ for some words $x',\lambda,\mu$; as $v'=y10x'0\lambda$, we would have $0x'0,1x'1\in\Ff(w)$, a contradiction since $w$ is balanced.

Similarly, $y$ is also the longest common prefix between $v'$ and the word $z'$ such that $\suc_{w}(z')=v'$. It follows $z=v'$ and $v=z'$, i.e.,
$v$ and $v'$ are lexicographically consecutive. Thus any two consecutive conjugates of a Christoffel word in $w$ satisfy the first condition in (3) of Proposition~\ref{thm:NFOp}.
More generally, if $z$ and $z'$ are lexicographically consecutive factors of $w$, then there exists a Christoffel factor
$0u1$ and two consecutive conjugates $v$ and $v'$ of $0u1$ with $z$ a prefix of $v$ and $z'$ a prefix of $v'$.
The result now follows. \end{proof}
Before proceeding to prove our main result, we need the following:

\begin{lem}\label{thm:sturmextension}
Let $w\in \{0,1\}^\omega$ be a Sturmian word and $x\in \{0,1\}$. If $xw$ satisfies NFOp then  $xw$ is Sturmian.
\end{lem}
\begin{proof}
We proceed by contradiction by supposing that $w$ is Sturmian, $xw$ satisfies NFOp and that $xw$ is not Sturmian. 
Without loss of generality we can assume that $x=0$. It follows that there exists $u$ such that both $0u0$ and $1u1$ are factors of $0w$.  Since $w$ is Sturmian, it follows from Proposition~\ref{thm:nonsturm} that $0u0$ is a unioccurrent prefix of $0w$ and every prefix of $0w$ is minimal in $0w$. On the other hand, if $\gamma$ denotes the characteristic word of $w$ (which has $u$ as a prefix), then every prefix of $0\gamma$ is minimal in $w$. By NFOp it follows that
for all $n>|u|+2$ the prefixes of length $n$ of $0w$ and $0\gamma$ can be written respectively as $0u01v_{n}$ and $0u10v_{n}$ for some word $v_{n}$. Hence
there exists a tail $v$ of $w$ such that $0w=0u01v$ and $0\gamma=0u10v$. Thus
$0v,1v\in\Omega_{w}$, so that $v=\gamma$ and hence $\gamma$ is a proper tail of itself, a contradiction.
\end{proof}

\begin{thm}\label{thm:main}
Let $w$ be an infinite word on the ordered alphabet $A=\{0,1,\ldots ,k\}$. The following statements are equivalent:
\begin{enumerate}
\item $w$ is Sturmian over the alphabet $\{0,1\}$.
\item $w$ satisfies NFOp.
\end{enumerate}
\end{thm}

\begin{proof}
That $(1)\Rightarrow (2)$ follows from Proposition~\ref{thm:firstDirection}.  
To see that $(2)\Rightarrow (1)$, we suppose that $w\in \{0,1\}^\omega$ satisfies NFOp and write $w=aw'$ with $a\in\{0,1\}$. We need to show that $w$ is Sturmian. By Lemma~\ref{thm:nonperiodic} $w$ is aperiodic. If $w$ is not Sturmian, then by Lemma~\ref{thm:sturmextension} we deduce that $w'$ is not Sturmian. Also,
combining Proposition~\ref{thm:nonsturm} and Lemma~\ref{thm:contradiction} we deduce that every prefix of $w$ is extremal and hence $w'$ also satisfies NFOp. In short, if $w=aw'$ satisfies NFOp and is not Sturmian, then every prefix of $w$ is extremal and the tail $w'$ satisfies NFOp and is not Sturmian. Thus writing $w'=bw''$ we deduce that every prefix of $w'$ is extremal and $w''$ satisfies NFOp and is not Sturmian. Iterating this process indefinitely we deduce that for each tail $v$ of $w$, each prefix of $v$ is extremal in $v$. 
Since $w$ is aperiodic it follows that there exists a tail $v$ of $w$ which begins in $01$ and a tail $v'$ of $v$ which begins in $10$. Since every prefix of $v$ is minimal in $v$ and every prefix of $v'$ is maximal in $v'$ it follows that
$00$ is not a factor of $v$ and $11$ is not a factor of $v'$. Hence $v'=(10)^\omega$, a contradiction.
\end{proof}
 
We next establish another characterization of Sturmian words based on the lexicographic order of their factors.

\begin{thm}
Let $w$ be a recurrent aperiodic binary word over the alphabet $\{ 0, 1 \}$.  Then the following are equivalent:
\begin{enumerate}
\item\label{it:sturm} $w$ is Sturmian.
\item\label{it:2pos} For all factors and $v, v' \in \Ff(w)$ of equal length, if $v' = \suc_{w}(v)$ then $v$ and $v'$ differ in at most two positions.
\item\label{it:v1} For all factors and $v, v' \in \Ff(w)$ of equal length, if $v \llex v'$ then  $|v|_{1} \leq |v'|_{1}$.
\end{enumerate}
\end{thm}

\begin{proof}
$(\ref{it:sturm}) \Rightarrow (\ref{it:2pos})$: Since $w$ is Sturmian, it has the NFOp by Theorem \ref{thm:main}. The statement is clearly proven since the NFOp trivially implies condition~(\ref{it:2pos}) by definition.

%

$(\ref{it:2pos}) \Rightarrow (\ref{it:v1})$: Notice that
condition~(\ref{it:2pos}) implies that if $f' = \suc_{w}(f)$, then there must exist $\lambda, \mu, \mu', x, x'$ with $|x|=|x'|\leq 1$ such that $f=\lambda 0\mu x\mu'$ and $f'=\lambda 1\mu x'\mu'$. Hence
\[
|f|_{1}=|\lambda|_{1}+|\mu|_{1}+|\mu'|_{1}+|x|_{1} \leq |\lambda|_{1}+|\mu|_{1}+|\mu'|_{1}+ 1 \leq 
 |\lambda|_{1}+|\mu|_{1}+|\mu'|_{1}+ |x'|_{1}+1 = |f'|_{1}.
\]
And thus, in particular $|f|_{1} \leq |f'|_{1}$.
Suppose $v \llex v'$. Then there must exist $v_{0}, \ldots , v_{k}$ such that $v=v_{0}$, $v'=v_{k}$ and for all $1 \leq n\leq k$, $v_{n} = \suc_{w}(v_{n-1})$, then 
\[ |v|_{1}=|v_{0}|_{1}\leq \cdots \leq |v_{k}|_{1}=|v'|_{1}. \]

$(\ref{it:v1}) \Rightarrow (\ref{it:sturm})$: Assume $w$ is not Sturmian; as $w$ is aperiodic, it has to be imbalanced. Since $w$ is recurrent, we have from Proposition \ref{thm:nonsturm} that there must exist $u$ such that both $10u0$ and $01u1$ are factors of $w$. But clearly this is a contradiction, since $01u1 \llex 10u0$ and  $|01u1|_{1}=|10u0|_{1}+1$. This concludes the proof. 
\end{proof}

Notice that, as opposed to Theorem \ref{thm:main}, the above result actually needs the recurrence and aperiodicity hypotheses, as for example:
\begin{itemize}
\item the recurrent periodic word $(01)^{\omega}$ and the non-recurrent aperiodic word $00f$ (where $f$ is the Fibonacci word) both respect condition~(\ref{it:v1}), although neither is Sturmian,
\item the non-recurrent ultimately periodic word $01^{\omega}$ satisfies both~(\ref{it:2pos}) and~(\ref{it:v1}), but it is not Sturmian.
\end{itemize}

\bibliographystyle{fundam}

\end{document}